\documentclass[11pt]{amsart}

\usepackage{latexsym}
\usepackage{amsmath}
\usepackage{amsfonts}
\usepackage{amssymb}
\usepackage{enumitem}
\usepackage {tikz}
\usepackage{color}
\usetikzlibrary {positioning}
\usetikzlibrary{arrows,shapes}
\usepackage{color}
\usepackage{graphicx}
\usepackage{latexsym}
\usepackage{soul}

\usepackage[noend]{algorithmic} 
\usepackage{algorithm,caption}

\algsetup{indent=2em} 
 
\usepackage{tikz}
\usepackage{color}
\usepackage{placeins}
\usetikzlibrary{positioning}
\usetikzlibrary{arrows,shapes}

\newtheorem{theorem}{Theorem}[section]
\newtheorem{lemma}[theorem]{Lemma}
\newtheorem{corollary}[theorem]{Corollary}
\newtheorem{proposition}[theorem]{Proposition}

\theoremstyle{definition}

\theoremstyle{remark}

\numberwithin{equation}{section}

\newcommand\Me{\mathcal{M}^\mathrm{ext}}

\newcommand\Md{\mathcal{M}^\mathrm{del}}

\newcommand\Meg{\mathcal{M}^\mathrm{extg}}

\newcommand\Ga{\mathcal{G}^\mathrm{epex}}

\newcommand\ba{\backslash}

\begin{document}

\title[Hereditary classes of matroids]{Extensions and Deletions of matroid classes closed under flats}

\author{Jagdeep Singh}
\address{Department of Mathematics and Statistics\\
Mississippi State University\\
Mississippi}
\email{singhjagdeep070@gmail.com}

\author{Vaidy Sivaraman}
\address{Department of Mathematics and Statistics\\
Mississippi State University\\
Mississippi}
\email{vaidysivaraman@gmail.com}


\date{\today}

\begin{abstract}
We call a class of matroids hereditary if it is closed under restriction to flats. For a hereditary class $\mathcal{M}$, its extension class consists of all matroids in $\mathcal{M}$ together with their single-element extensions. The deletion class consists of all matroids in $\mathcal{M}$ along with their single-element deletions.

We prove that if $\mathcal{M}$ has finitely many forbidden flats, then the forbidden flats for its extension class have bounded rank. For $GF(q)$-representable matroids where $q$ is in $\{2,3\}$, we exploit correspondence with $2$-colorings of projective geometries to establish the analogous result for the deletion class. We also note the consequences for hereditary classes of graphs, discussing the interplay of graphs and matroids.
\end{abstract}

\maketitle

\section{Introduction}
\label{intro}
We consider simple matroids and simple graphs. The notation and terminology follow~\cite{ox2}. An {\bf induced subgraph} of a graph $G$ is a graph $H$ obtained from $G$ by a sequence of vertex deletions. A class $\mathcal{G}$ of graphs is called {\bf hereditary} if it is closed under taking induced subgraphs. For a hereditary class $\mathcal{G}$, we call a graph $H$ a {\bf forbidden induced subgraph} for $\mathcal{G}$ if $H$ is not in $\mathcal{G}$ but every proper induced subgraph of $H$ is in $\mathcal{G}$. 

A class $\mathcal{M}$ of matroids is called {\bf hereditary} if it is closed under restriction to flats. For a hereditary class $\mathcal{M}$ of matroids, we call a matroid $N$ a {\bf forbidden flat} for $\mathcal{M}$ if $N$ is not in $\mathcal{M}$, but for every proper flat $F$ of $N$, the restriction $N|F$ is in $\mathcal{M}$. We say that a matroid $M$ {\it contains} $N$ as a flat when $M$ has a flat $F$ such that the restriction $M|F$ is isomorphic to $N$. Note that a matroid $M$ is in $\mathcal{M}$ if and only if $M$ contains no forbidden flat for $\mathcal{M}$. 

Hereditary classes of matroids arise naturally from hereditary classes of graphs. For a hereditary class $\mathcal{G}$ of graphs, consider the class $\mathcal{M}$ of cycle matroids of the graphs in $\mathcal{G}$. For some hereditary classes $\mathcal{G}$ such as the class of cographs (graphs not containing the four-vertex path as an induced subgraph), the class $\mathcal{M}$ is hereditary. In Section~$4$, we show that if a hereditary class $\mathcal{G}$ of graphs is closed under disjoint union, then the corresponding matroid class $\mathcal{M}$ is hereditary.  

While every minor-closed class of matroids is hereditary, we focus on hereditary classes defined by finitely many forbidden flats. A key example is the class of targets. Let $M$ be a rank-$r$ projective geometry represented over $GF(q)$, the field with $q$ elements. We call $(F_0, F_1, \ldots, F_k)$ \textit {a nested sequence of projective flats} if $\emptyset = F_0 \subseteq F_1 \subseteq \ldots \subseteq F_{k-1} \subseteq F_k = E(M)$ and each $F_i$ is a, possibly empty, flat of $M$. For a subset $X$ of $E(PG(r-1,q))$, we call $PG(r-1,q)|X$ a \textbf{target}, if there is a nested sequence $(F_0, F_1, \ldots , F_k)$ of projective flats such that $X$ is the union of all sets $F_{i+1} - F_i$ for $i$ even.

Note that the class of $GF(q)$-targets is hereditary. The following theorems characterize targets via a finite set of forbidden flats. 

\begin{theorem}[\cite{nelson}]
\label{nelson_nomoto}
Let $(G,R)$ be a $2$-coloring of $PG(r-1,2)$.  Then $PG(r-1,2)|G$ is a target if and only if it does not contain $U_{3,3}$ or $U_{2,3} \oplus U_{1,1}$ as a flat.
\end{theorem}

\begin{theorem}[\cite{ox1}]
\label{mizzel}
For a prime power $q$ exceeding two, let $(G,R)$ be a $2$-coloring of $PG(r-1,q)$.  Then $PG(r-1,q)|G$ is a target if and only if it does not contain any of $U_{2,2}, U_{2,3}, \ldots, U_{2,q-2},$ or $U_{2,q-1}$ as a flat.
\end{theorem}

Our main contribution concerns the stability of such characterizations under the operations of matroid extension and deletion. The {\bf edge-apex class, $\Ga$,} for a hereditary class $\mathcal{G}$ of graphs is the class of graphs $G$ such that $G$ is in $\mathcal{G}$, or $G$ has an edge $e$ such that $G-e$ is in $\mathcal{G}$.
Similarly, the {\bf extension class, $\Me$,} of a hereditary class $\mathcal{M}$ of matroids is the class of matroids $M$ such that $M$ is in $\mathcal{M}$, or $M$ contains an element $e$ such that $M \ba e$ is in $\mathcal{M}$. 

In \cite{sinsiv}, we showed that if a hereditary class $\mathcal{G}$ of graphs has finitely many forbidden induced subgraphs, then so does $\Ga$. In particular, we proved the following.

\begin{theorem}
\label{unique_FIS}
Let $\mathcal{G}$ be a hereditary class of graphs and let $c$ and $k$ denote the maximum number of vertices and edges in a forbidden induced subgraph for $\mathcal{G}$. If $G$ is a forbidden induced subgraph for $\Ga$, then $|V(G)| \leq$  $\max \{2c, c + k(c-2)\}$. 
\end{theorem}

Here we prove its matroid analogue. In particular, we show that for a hereditary class $\mathcal{M}$ of matroids, its extension class $\Me$ is hereditary, and if $\mathcal{M}$ has a finite number of forbidden flats, then the rank of forbidden flats for $\Me$ is bounded. Following is the precise statement. 

\begin{theorem}
\label{bound_k}
Let $\mathcal{M}$ be a hereditary class of matroids and let $r$ denote the maximum rank of a forbidden flat for  $\mathcal{M}$, and $k$ denote the maximum number of elements in a forbidden flat for  $\mathcal{M}$. Then $\Me$ is hereditary. Furthermore, every forbidden flat for $\Me$ has rank at most $\max \{2r, r + k(r-1)\}$.
\end{theorem}

The \textbf{deletion class}, $\Md$, of a hereditary class $\mathcal{M}$ of matroids is the class of matroids $M$ such that $M$ is in $\mathcal{M}$, or $M \cong N \ba e$ for a matroid $N$ in $\mathcal{M}$ and an element $e$ of $N$. We exploit the complementary relationship between extension and deletion in $GF(q)$-representable matroids for $q$ in $\{2, 3\}$, and prove an analogous result for $\Md$ in Section~$3$.

In Section $4$, we show that the relationship between graph and matroid hereditary classes allows us to derive a version of Theorem~\ref{unique_FIS} as a direct consequence of our matroid results.

\section{Exclusions for the extension class}

\begin{lemma}
\label{hereditary}
$\Me$ is hereditary.
\end{lemma}
\begin{proof}
Let $M$ be a matroid in $\Me$ and let $F$ be a flat of $M$. If $M$ is in $\mathcal{M}$, then the restriction $M|F$ is in $\mathcal{M}$, and so $M|F$ is in $\Me$. Otherwise, $M$ has an element $e$ such that $M \ba e$ is in $\mathcal{M}$. If $e$ is not in $F$, then $F$ is a flat of $M \ba e$ so  $M|F$ is in $\Me$. Thus $e$ is in $F$.  It follows that $(M|F)\ba e$ is in $\mathcal{M}$ so $M|F$ is in $\Me$.
\end{proof}

We prove a slightly stronger technical version of Theorem \ref{bound_k}, which provides a sharper bound on the rank of forbidden flats, and implies Theorem \ref{bound_k}. 

\begin{theorem}
\label{bound_k_general}
Let $\mathcal{M}$ be a hereditary class of matroids, and let $\mathcal{F}$ be the set of its forbidden flats. Let $r$ denote the maximum rank of a flat in $\mathcal{F}$. If $M$ is a forbidden flat for $\Me$, then the rank of $M$ is at most $\max\limits_{F \in \mathcal{F}}{\{2r, r(F) + |F|(r-1)\}}$. 
\end{theorem}

\begin{proof}
Note that $M$ is not in $\mathcal{M}$ so $E(M)$ has a flat $F_0$ such that the restriction $M|F_0$ is a forbidden flat for $\mathcal{M}$. Suppose that $F_0 = \{e_1, \ldots, e_k\}$ and $r(F_0) = s$. Since for each $i$ in $\{1, \ldots, k\}$, the matroid $M \ba e_i$ is not in $\mathcal{M}$, we have a subset $F_i$ of $E(M)-e_i$ such that the restriction $(M \ba e_i)|F_i$ is a forbidden flat for $\mathcal{M}$. Observe that, if $\bigcup_{i=0}^{k} F_i$ is contained in a hyperplane $H$ of $M$, then the matroid $M|H$ is not in $\Me$, a contradiction to the minimality of $M$. Therefore, $r(M) = r(\bigcup_{i=0}^{k} F_i)$.

First suppose that the intersection of a pair in $\{F_0, F_1, \ldots, F_k\}$, say $F_m$ and $F_n$, is empty. Then $r(M)=r(F_m \cup F_n)$. If not, then $F_m \cup F_n$ is contained in a hyperplane $H$ of $M$, and $M|H$ is not in $\Me$, a contradiction. Since $r(F_m) \leq r$ and $r(F_n) \leq r$, it follows that $r(M) = r(F_m\cup F_n) \leq 2r$. 

Therefore we may assume that each such pair has a non-empty intersection. It now follows from the submodularity of the rank function that $r(\bigcup_{i=0}^{k} F_i) \leq s + k(r-1)$. Since $s = r(F_0)$ and $k = |F_0|$ for the forbidden flat $F_0 \in \mathcal{F}$, we conclude that $r(M) \leq \max\limits_{F \in \mathcal{F}} \{2r, r(F) + |F|(r-1)\}$. This completes the proof.
\end{proof}

\section{Deletion class}

Recall that the deletion class of a hereditary class $\mathcal{M}$ of matroids is the class of matroids $M$ such that $M$ is in $\mathcal{M}$, or $M$ is a single-element deletion of a matroid in $\mathcal{M}$. It is clear that $\Md$ is hereditary. In this section, we show that if $\mathcal{M}$ is a hereditary class of $GF(q)$-representable matroids with finitely many forbidden flats, where $q$ is in $\{2, 3\}$, then the same holds for $\Md$. 

All matroids considered in this section are $GF(q)$-representable matroids, where $q$ is in $\{2, 3\}$. We specifically restrict our attention to binary and ternary matroids to use their property of unique representability. A \textbf{$2$-coloring} of a projective geometry $PG(r-1,q)$ is a partition $(G,R)$ of its points. The elements of $G$ are called the \textit{green points} of $PG(r-1,q)$, and the elements of $R$ the \textit{red points}. Note that the restriction $PG(r-1,q)|G$ of the projective geometry to the green points is a representation of a matroid $M$ over the field $GF(q)$. For a $GF(q)$-representable matroid $M$, we consider all $2$-colorings $(G,R)$ of projective geometries of rank $r \ge r(M)$ such that $PG(r-1,q)|G \cong M$. These are the $2$-colorings \textit{representing} $M$. The \textit{rank} of a $2$-coloring $C=(G,R)$ of $PG(r-1,q)$ is the rank of the underlying projective geometry $PG(r-1,q)$. Note that this is the rank of the ambient space and is not necessarily the minimal rank into which the matroid embeds.

We call the flats of $PG(r-1,q)$ \textit{projective flats}. A class $\mathcal{C}$ of $2$-colorings of projective geometries over $GF(q)$ is \textbf{hereditary} if the restriction of a coloring in $\mathcal{C}$ to any projective flat is again in $\mathcal{C}$. A coloring $C$ not in $\mathcal{C}$ whose restriction to every proper projective flat is in $\mathcal{C}$ is a \textbf{forbidden $2$-coloring} for a hereditary class $\mathcal{C}$.
 For a class $\mathcal{N}$ of $GF(q)$-representable matroids, let $\mathcal{N}_c$ denote the collection of all $2$-colorings of projective geometries over $GF(q)$ representing matroids in $\mathcal{N}$. We call $\mathcal{N}_c$ the $2$-coloring class corresponding to $\mathcal{N}$. Observe that if $\mathcal{N}$ is hereditary, then so is $\mathcal{N}_c$.

\begin{lemma}
\label{forbidden_colorings_flats}
For $q$ in $\{2, 3\}$, let $\mathcal{M}$ be a hereditary class of $GF(q)$-representable matroids, and let $\mathcal{M}_c$ be its corresponding $2$-coloring class. Then a matroid $M$ is a forbidden flat for $\mathcal{M}$ if and only if every $2$-coloring of $PG(r(M)-1,q)$ representing $M$ is a forbidden $2$-coloring for $\mathcal{M}_c$. Moreover, no $2$-coloring of rank greater than $r(M)$ representing $M$ is forbidden for $\mathcal{M}_c$. 
\end{lemma}

\begin{proof}
Let $C = (G,R)$ be a $2$-coloring of $PG(r(M)-1,q)$ representing a matroid $M$. If $M$ is a forbidden flat for $\mathcal{M}$, then $M$ is not in $\mathcal{M}$ so $C$ is not in $\mathcal{M}_c$. However, the restriction of $C$ to any proper projective flat is in $\mathcal{M}_c$ since it corresponds to a proper flat of $M$, which is in $\mathcal{M}$. It follows that $C$ is a forbidden $2$-coloring for $\mathcal{M}_c$. Conversely, suppose that $C$ is a forbidden $2$-coloring for $\mathcal{M}_c$. Then $C$ is not in $\mathcal{M}_c$ so $M$ is not in $\mathcal{M}$. Since every proper flat of $M$ corresponds to a restriction of $C$, which is in $\mathcal{M}_c$, it follows that $M$ is a forbidden flat for $\mathcal{M}$.

Finally, suppose that $C'$ is a coloring of rank greater than $r(M)$ that represents $M$. Then there is a restriction of $C'$ to a proper projective flat that represents $M$ and so is not in $\mathcal{M}_c$. Therefore $C'$ is not a forbidden $2$-coloring for $\mathcal{M}_c$. 
\end{proof}

For a coloring $C=(G,R)$ of $PG(r-1,q)$, its \textit{complement}, $\mathrm{comp}(C) := (R,G)$, is obtained by swapping the red and green points. For $q$ in $\{2, 3\}$ and a class $\mathcal{C}$ of $2$-colorings of projective geometries over $GF(q)$, its complement is
\[
\mathrm{comp}(\mathcal{C}) := \{\mathrm{comp}(C) \mid C \in \mathcal{C}\}.
\]

We omit the straightforward proof of the next result.

\begin{lemma}
\label{complement_class}
Let $\mathcal{C}$ be a class of $2$-colorings of projective geometries over $GF(q)$, where $q$ is in  $\{2, 3\}$. 
If $\mathcal{C}$ is hereditary, then its complementary class $\mathrm{comp}(\mathcal{C})$ is also hereditary. Moreover, a $2$-coloring $C$ is forbidden for $\mathcal{C}$ if and only if $\mathrm{comp}(C)$ is forbidden for $\mathrm{comp}(\mathcal{C})$.
\end{lemma}

For $q$ in $\{2, 3\}$ and a class $\mathcal{C}$ of 2-colorings of projective geometries over $GF(q)$, its \textit{extension class}, $\mathcal{C}^{\rm extq}$, consists of colorings in $\mathcal{C}$ together with those obtained from a coloring in $\mathcal{C}$ by flipping a single red point to green. The \textit{deletion class}, $\mathcal{C}^{\rm del}$, of $\mathcal{C}$ consists of colorings in $\mathcal{C}$ along with those obtained from a coloring in $\mathcal{C}$ by flipping a single green point to red.

For a hereditary class $\mathcal{M}$ of $GF(q)$-representable matroids where $q$ is in $\{2, 3\}$, we denote by $\mathcal{M}^{\rm extq}$ the class of all $GF(q)$-representable matroids in $\Me$. The next result follows from the unique representability of binary and ternary matroids.

\begin{lemma}
\label{commutative}
Let $\mathcal{M}$ be a hereditary class of $GF(q)$-representable matroids, where $q$ is in $\{2, 3\}$, and let $\mathcal{M}_c$ be its corresponding $2$-coloring class. Then the $2$-coloring class, $(\mathcal{M}^{\rm del})_c$, corresponding to $\mathcal{M}^{\rm del}$ equals the deletion class $(\mathcal{M}_c)^{\rm del}$ of $\mathcal{M}_c$. Also, $(\mathcal{M}^{\rm extq})_c$ equals the extension class $(\mathcal{M}_c)^{\rm extq}$ of $\mathcal{M}_c$.
\end{lemma}

\begin{lemma}
\label{deletion_extension_complementation}
Let $\mathcal{M}$ be a hereditary class of $GF(q)$-representable matroids, where $q$ is in $\{2, 3\}$, and let $\mathcal{M}_c$ be its corresponding $2$-coloring class. Then
\[
\mathrm{comp}((\mathcal{M}_c)^{\rm del}) = (\mathrm{comp}(\mathcal{M}_c))^{\rm extq}.
\]
Consequently, a coloring $C$ is forbidden for $(\mathcal{M}_c)^{\rm del}$ if and only if $\mathrm{comp}(C)$ is forbidden for $(\mathrm{comp}(\mathcal{M}_c))^{\rm extq}$.
\end{lemma}

\begin{proof}
By definition, a coloring $C$ is in $(\mathcal{M}_c)^{\rm del}$ if $C \in \mathcal{M}_c$, or $C$ is obtained from a coloring in $\mathcal{M}_c$ by flipping a single green point to red. Note that if $C \in \mathcal{M}_c$, then $\mathrm{comp}(C)$ is in  $\mathrm{comp}(\mathcal{M}_c)$. Moreover, if $C$ is obtained by flipping a green point to red from a coloring in $\mathcal{M}_c$, then $\mathrm{comp}(C)$ is obtained from a coloring in $\mathrm{comp}(\mathcal{M}_c)$ by flipping a red point to green. It follows that $\mathrm{comp}(C)$ is in the extension class of $\mathrm{comp}(\mathcal{M}_c)$. Hence every coloring in $\mathrm{comp}((\mathcal{M}_c)^{\rm del})$  lies in $(\mathrm{comp}(\mathcal{M}_c))^{\rm extq}$.

Conversely, every coloring in $(\mathrm{comp}(\mathcal{M}_c))^{\rm extq}$ arises as the complement of some coloring in $(\mathcal{M}_c)^{\rm del}$. This establishes the equality of the two classes. The statement about forbidden colorings follows immediately.
\end{proof}

We now apply the bound from Theorem \ref{bound_k} to the coloring context. 

\begin{lemma}
\label{coloring_bound_k}
Let $\mathcal{C}$ be a hereditary class of $2$-colorings of projective geometries over $GF(q)$. Suppose that every forbidden $2$-coloring for $\mathcal{C}$ has rank at most $r$, and has at most $k$ green elements. If $C=(G,R)$ is a forbidden $2$-coloring for the extension class $\mathcal{C}^{\rm extq}$, then the rank of $C$ is at most $\max\{ 2r,\; r + k(r-1) \}.$
\end{lemma}

\begin{proof}
The proof follows exactly the same structure as the proof of Theorem \ref{bound_k_general}. Let $C=(G,R)$ be a forbidden $2$-coloring for $\mathcal{C}^{\rm extq}$ with ambient space $PG(r(C)-1, q)$. Since $C$ is not in $\mathcal{C}$, it has a projective flat $P_0$ such that the restriction $C|P_0$ is a forbidden coloring for $\mathcal{C}$. Let the green points of $C|P_0$ be $\{e_1, \ldots, e_k\}$ and $r(P_0) = s$. For each $i$ in $\{1, \ldots, k\}$, the coloring obtained by flipping $e_i$ to red is not in $\mathcal{C}$, so it has a projective flat $P_i$ such that its restriction to $P_i$ is forbidden for $\mathcal{C}$. By the submodularity of the rank function of the projective geometry, the rank of the union of these projective flats is at most $s + k(r-1)$. Following the exact hyperplane contradiction argument from Theorem \ref{bound_k_general}, we conclude that the rank of $C$ is at most $\max\{ 2r, r + k(r-1) \}$.
\end{proof}

\begin{theorem}
\label{deletion_bound_k_corrected}
Let $\mathcal{M}$ be a hereditary class of $GF(q)$-representable matroids, where $q$ is in $\{2, 3\}$. Let $r$ denote the maximum rank of a forbidden flat for $\mathcal{M}$, and let $k'$ denote the maximum number of red points in a forbidden $2$-coloring for $\mathcal{M}_c$. If $M$ is a forbidden flat for the deletion class $\Md$, then $r(M) \le \max\{ 2r, r + k'(r-1) \}$.
\end{theorem}

\begin{proof}
Let $C$ be a $2$-coloring of $PG(r(M)-1,q)$ representing $M$. By Lemma \ref{forbidden_colorings_flats}, $C$ is a forbidden $2$-coloring for $(\Md)_c$. By Lemma \ref{commutative}, $(\Md)_c$ equals $(\mathcal{M}_c)^{\rm del}$, so $C$ is a forbidden $2$-coloring for $(\mathcal{M}_c)^{\rm del}$. Further, by Lemmas \ref{complement_class} and \ref{deletion_extension_complementation}, it follows that $\mathrm{comp}(C)$ is a forbidden $2$-coloring for $(\mathrm{comp}(\mathcal{M}_c))^{\rm extq}$. 

Since the rank of the colorings $C$ and $\mathrm{comp}(C)$ both equal $r(M)$, it is enough to show that the rank of a forbidden $2$-coloring for the $2$-coloring class $(\mathrm{comp}(\mathcal{M}_c))^{\rm extq}$ is at most $\max\{ 2r, r + k'(r-1) \}$. 

By Lemma \ref{complement_class}, since $\mathcal{M}_c$ is a hereditary class of $2$-colorings, its complement $\mathrm{comp}(\mathcal{M}_c)$ is also a hereditary class of $2$-colorings. Observe that $r$ is the maximum rank and $k'$ is the maximum number of green points in a forbidden $2$-coloring for $\mathrm{comp}(\mathcal{M}_c)$. By applying Lemma \ref{coloring_bound_k} to the coloring class $\mathrm{comp}(\mathcal{M}_c)$, it follows that a forbidden $2$-coloring for $(\mathrm{comp}(\mathcal{M}_c))^{\rm extq}$ has rank at most $\max\{ 2r, r + k'(r-1) \}$. 
\end{proof}

\begin{proposition}
\label{union_of_hereditary_matroids}
For $q$ in $\{2,3\}$, let $\mathcal{M}_1$ and $\mathcal{M}_2$ be hereditary classes of $GF(q)$-representable matroids. Then, the set-theoretic union $\mathcal{M}_1 \cup \mathcal{M}_2$ is hereditary. Moreover, if $c$ denotes the maximum rank of a forbidden flat for $\mathcal{M}_1$, and $d$ denotes the maximum rank of a forbidden flat for $\mathcal{M}_2$, then every forbidden flat for $\mathcal{M}_1 \cup \mathcal{M}_2$ has rank at most $c+d$. 
\end{proposition}

\begin{proof}
It is clear that $\mathcal{M}_1 \cup \mathcal{M}_2$ is hereditary. Let $M$ be a forbidden flat for $\mathcal{M}_1 \cup \mathcal{M}_2$. Since $M$ is not in $\mathcal{M}_1$, it has a flat $F_1$ such that $M|F_1$ is a forbidden flat for $\mathcal{M}_1$. Similarly, it has a flat $F_2$ such that $M|F_2$ is a forbidden flat for $\mathcal{M}_2$. Since every hyperplane of $M$ is in $\mathcal{M}_1 \cup \mathcal{M}_2$, it follows that no hyperplane of $M$ contains $F_1 \cup F_2$. Therefore the rank of $M$ equals the rank of $M|(F_1 \cup F_2)$, which is at most $r(F_1) + r(F_2) \leq c+d$. 
\end{proof}

For $q$ in $\{2,3\}$, and a hereditary class $\mathcal{M}$ of $GF(q)$-representable matroids, the \textbf{almost-$\mathcal{M}$ class} is the set-theoretic union $\Me \cup \Md$. The following is an immediate consequence of Theorems \ref{bound_k}, \ref{deletion_bound_k_corrected}, and Proposition \ref{union_of_hereditary_matroids}.

\begin{corollary}
\label{matroid_corollary}
For $q$ in $\{2,3\}$, let $\mathcal{M}$ be a hereditary class of $GF(q)$-representable matroids such that $\mathcal{M}$ has a finite set of forbidden flats. Then the almost-$\mathcal{M}$ class is hereditary, and has a finite set of forbidden flats. 
\end{corollary}

\section{Flat-hereditary graph classes}

Note that if all matroids in a hereditary class $\mathcal{M}$ of matroids are graphic, its extension class may contain non-graphic matroids as well. However by restricting the extension class to contain only graphic matroids we still retain the property of being hereditary. For a hereditary class $\mathcal{M}$ of graphic matroids, a matroid $M$ is in the \textbf{graphic extension class}, $\Meg$, of $\mathcal{M}$ if $M$ is graphic and $M$ is in the extension class $\Me$ of $\mathcal{M}$. While working with only graphic matroids, it makes sense to consider only those forbidden flats that are graphic. For a hereditary class $\mathcal{M}$ of graphic matroids, we say $F$ is a \textbf{graphic forbidden flat} for $\mathcal{M}$ if $F$ is graphic and a forbidden flat for $\mathcal{M}$. The following is the statement of Theorem \ref{bound_k} restricted to graphic matroids.

\begin{corollary}
\label{bound_k_corollary}
Let $\mathcal{M}$ be a hereditary class of graphic matroids and let $r$ denote the maximum rank of a graphic forbidden flat for $\mathcal{M}$, and $k$ denote the maximum number of elements in a graphic forbidden flat for $\mathcal{M}$. If $M$ is a graphic forbidden flat for $\Meg$, then the rank of $M$ is at most $\max\{2r, r + k(r-1)\}$. 
\end{corollary}

For a graph $G$, a set of edges is a flat of the cycle matroid $M(G)$ if and only if the subgraph induced by those edges has the property that each of its connected components is an induced subgraph of $G$. Inspired by this observation, we say a subgraph $F$ of a graph $G$ is a \textbf{flat} of $G$ if $F$ is a disjoint union of induced subgraphs of $G$. We call a class of graphs $\mathcal{G}$ \textbf{flat-hereditary} if $\mathcal{G}$ is closed under taking flats. A graph $H$ is a \textbf{forbidden flat} for $\mathcal{G}$ if $H$ is not in $\mathcal{G}$ but every proper flat of $H$ is in $\mathcal{G}$. 

Let $G_1$ and $G_2$ be graphs. If their vertex sets are disjoint, their {\bf $0$-sum}, $G_1 \oplus G_2$, is their disjoint union. Now suppose that $|V(G_1) \cap V(G_2)| = 1$. Then the union of $G_1$ and $G_2$, which has vertex set $V(G_1) \cup V(G_2)$ and edge set $E(G_1) \cup E(G_2)$, is a {\bf $1$-sum}, $G_1 \oplus_1 G_2$, of $G_1$ and $G_2$.

The following proposition notes that a hereditary class closed under $0$-sum is flat-hereditary.

\begin{proposition}
\label{prop_b}
Let $\mathcal{G}$ be a hereditary class of graphs closed under $0$-sum. Then $\mathcal{G}$ is flat-hereditary. Moreover, $H$ is a forbidden induced subgraph for $\mathcal{G}$ if and only if $H$ is a forbidden flat for $\mathcal{G}$.
\end{proposition}

\begin{proof}

For a graph $G$ in $\mathcal{G}$, let $F$ be a flat of $G$. Then $F = F_1 \oplus \ldots \oplus F_k$, where each $F_i$ is an induced subgraph of $G$. Since each $F_i$ is in $\mathcal{G}$ and the class $\mathcal{G}$ is closed under $0$-sum, the graph $F$ is in $\mathcal{G}$. Therefore $\mathcal{G}$ is flat-hereditary. 

Let $H$ be a forbidden induced subgraph for $\mathcal{G}$ and let $T$ be a proper flat of $H$. Then $T = T_1 \oplus \ldots \oplus T_k$ where each $T_i$ is a proper induced subgraph of $H$. Since each $T_i$ is in $\mathcal{G}$, the graph $T$ is in $\mathcal{G}$. It follows that $H$ is a forbidden flat for $\mathcal{G}$. Conversely, if $H$ is a forbidden flat for $\mathcal{G}$, it is straightforward to see that $H$ is a forbidden induced subgraph for $\mathcal{G}$ as well.
\end{proof}

Instead of verifying closure under $0$-sum of a hereditary class, it is enough to check if all its forbidden induced subgraphs are connected due to the following.

\begin{lemma}
\label{prelim}
Let $\mathcal{G}$ be a hereditary class of graphs. Then all forbidden induced subgraphs for $\mathcal{G}$ are connected if and only if $\mathcal{G}$ is closed under $0$-sum.
\end{lemma}

\begin{proof}
Let $G_1$ and $G_2$ be graphs in $\mathcal{G}$. Note that both $G_1$ and $G_2$ contain no forbidden induced subgraph for $\mathcal{G}$. Since every forbidden induced subgraph for $\mathcal{G}$ is connected, it follows that the disjoint union $G_1 \oplus G_2$ contains no forbidden induced subgraph for $\mathcal{G}$. It follows that $G_1 \oplus G_2$ is in $\mathcal{G}$. 

Conversely, let $H$ be a forbidden induced subgraph for $\mathcal{G}$ such that $H = H_1 \oplus H_2$. Since $H_1$ and $H_2$ are proper induced subgraphs of $H$, both are in $\mathcal{G}$. Since $\mathcal{G}$ is closed under $0$-sum, it follows that $H$ is in $\mathcal{G}$, a contradiction. Therefore $H$ is connected.
\end{proof}

The following characterizes the hereditary classes whose forbidden induced subgraphs are $2$-connected. We omit its proof as the proof follows the same outline as the proof of Lemma \ref{prelim}.

\begin{lemma}
\label{prelim_stronger}
Let $\mathcal{G}$ be a hereditary class of graphs. Then all forbidden induced subgraphs for $\mathcal{G}$ are $2$-connected if and only if $\mathcal{G}$ is closed under $0$-sum and $1$-sum.
\end{lemma}

It follows by Proposition \ref{prop_b} that the closure under $0$-sum of a hereditary class is sufficient to imply that it is flat-hereditary. However it is not necessary that a flat-hereditary class is closed under $0$-sum. The following illustrates that for a hereditary class, being flat-hereditary is weaker than closure under $0$-sum.

\begin{lemma}
\label{characterization_flat_hereditary}
Let $\mathcal{G}$ be a hereditary class of graphs. Then $\mathcal{G}$ is flat-hereditary if and only if for any pair $H_1, H_2$ of graphs in $\mathcal{G}$ that are induced subgraphs of a graph $H$ in $\mathcal{G}$ such that $V(H_1) \cap V(H_2)$ is empty, the graph $H_1 \oplus H_2$ is in $\mathcal{G}$. 
\end{lemma}

\begin{proof}
 Observe that $H_1 \oplus H_2$ is a flat of $H$. Since $H$ is in $\mathcal{G}$ and $\mathcal{G}$ is flat-hereditary, it follows that $H_1 \oplus H_2$ is in $\mathcal{G}$. Conversely, let $F$ be a flat of a graph $G$ in $\mathcal{G}$. Then $F = F_1 \oplus \ldots \oplus F_k$ where each $F_i$ is an induced subgraph of $G$. It follows that $F$ is in $\mathcal{G}$ so $\mathcal{G}$ is flat-hereditary.   
\end{proof}

It is clear that a flat-hereditary class $\mathcal{G}$ is hereditary and each of its forbidden flats is a forbidden induced subgraph as well. The following notes that every forbidden induced subgraph for $\mathcal{G}$ contains a forbidden flat for $\mathcal{G}$ of the same order.

\begin{lemma}
\label{flat_fis}
Let $\mathcal{G}$ be a flat-hereditary class of graphs. Then every forbidden induced subgraph $G$ of $\mathcal{G}$ has a flat $F$ such that $F$ is a forbidden flat for $\mathcal{G}$ and $V(G) = V(F)$. In particular, if $\alpha$ is the maximum number of vertices in a forbidden flat for $\mathcal{G}$ and $\beta$ is the maximum number of vertices in a forbidden induced subgraph for $\mathcal{G}$, then $\alpha = \beta$.
\end{lemma}

\begin{proof}
Since $G$ is not in $\mathcal{G}$, the graph $G$ contains a flat $F$ such that $F$ is a forbidden flat for $\mathcal{G}$. Note that for each vertex $v$ of $G$, the graph $G-v$ is in $\mathcal{G}$ so $G-v$ does not contain $F$ as a flat. It follows that $V(G)=V(F)$. Since every forbidden flat for $\mathcal{G}$ is a forbidden induced subgraph for $\mathcal{G}$, it follows that $\alpha = \beta$. 
\end{proof}

For a flat-hereditary class $\mathcal{G}$ of graphs, consider the class $\mathcal{M}$ of cycle matroids of graphs in $\mathcal{G}$. Note that $\mathcal{M}$ is a hereditary class of matroids. However the  graphic forbidden flats for $\mathcal{M}$ are not necessarily the cycle matroids of the forbidden flats for $\mathcal{G}$. The failure of having a correspondence between the forbidden flats for $\mathcal{G}$ and the graphic forbidden flats for $\mathcal{M}$ is due to the fact that distinct graphs may have the same cycle matroid. In particular, if two graphs $G_1$ and $G_2$ have the same cycle matroid $M$, and $G_1$ is in $\mathcal{G}$ while $G_2$ is a forbidden flat for $\mathcal{G}$, then $M$ is clearly not a forbidden flat for $\mathcal{M}$. Therefore it is not possible to apply Theorem \ref{bound_k} to the class of cographs, for instance. Let $\mathcal{G}$ be the class of cographs and let $\mathcal{M}$ be the class of cycle matroids of cographs. Since cographs are closed under $0$-sum, the class $\mathcal{G}$ is flat-hereditary and by Proposition \ref{prop_b}, the path of length three, $P_4$, is the unique forbidden flat for $\mathcal{G}$. However any circuit of size at least five is a forbidden flat for $\mathcal{M}$. Also the cycle matroid, $U_{3,3}$, of $P_4$ is in $\mathcal{M}$.

We call a class $\mathcal{G}$ of graphs \textbf{matroid-closed} if for any graph $G$ in $\mathcal{G}$, all graphs $H$ that have the same cycle matroid as $G$ are also in $\mathcal{G}$. Note that the class of cographs is not matroid-closed. In order to be able to apply Theorem \ref{bound_k} to graphs, we need flat-hereditary classes of graphs that are matroid-closed. Let $v$ and $v'$ be vertices of distinct components of a graph $G$. A \textbf{vertex identification} of $G$ modifies $G$ by identifying $v$ and $v'$ as a new vertex $\overline{v}$. The reverse operation of vertex identification is \textbf{vertex cleaving}. Suppose that a graph $G$ is obtained from disjoint graphs $G_1$ and $G_2$ by identifying the vertices $u_1$ of $G_1$ and $u_2$ of $G_2$ as the vertex $u$ of $G$, and identifying the vertices $v_1$ of $G_1$ and $v_2$ of $G_2$ as the vertex $v$ of $G$. In a \textbf{twisting} of $G$ about $\{u,v\}$, we identify, instead, $u_1$ with $v_2$ and $v_1$ with $u_2$. Whitney's $2$-isomorphism theorem (see \cite{ox2}) characterizes when two graphs have the same cycle matroids. The following is its restatement.

\begin{theorem}
\label{whitney}
Let $\mathcal{G}$ be a class of graphs. Then $\mathcal{G}$ is matroid-closed if and only if $\mathcal{G}$ is closed under the operations of vertex cleaving, vertex identification, and twisting.
\end{theorem}

\begin{lemma}
\label{error_fixed}
Let $\mathcal{G}$ be a flat-hereditary class of graphs that is closed under the operations of vertex cleaving, vertex identification, and twisting, and let $\mathcal{M}$ be the class of cycle matroids of graphs in $\mathcal{G}$. Then $\mathcal{M}$ is hereditary. Moreover, $F$ is a graphic forbidden flat for $\mathcal{M}$ if and only if $F$ is the cycle matroid of a forbidden flat for $\mathcal{G}$. 
\end{lemma}

\begin{proof}
Let $M$ be a matroid in $\mathcal{M}$ and let $F$ be a flat of $M$. Note that $M$ is the cycle matroid of a graph $G$ in $\mathcal{G}$ and $F$ is the cycle matroid of a flat $H$ of $G$. Since $\mathcal{G}$ is flat-hereditary, it follows that $H$ is in $\mathcal{G}$ so $F$ is in $\mathcal{M}$. 

Suppose that $F$ is a graphic forbidden flat for $\mathcal{M}$ and $G$ is a graph whose cycle matroid is isomorphic to $F$. We show that $G$ is a forbidden flat for $\mathcal{G}$. Clearly $G$ is not in $\mathcal{G}$. Let $H$ be a proper flat of $G$. Then the cycle matroid of $H$ is a proper flat $F_H$ of $F$. Note that $F_H$ is in $\mathcal{M}$ so $\mathcal{G}$ contains a graph $H'$ whose cycle matroid is isomorphic to $F_H$. By Theorem \ref{whitney}, the graph $H$ can be obtained from $H'$ via the operations of vertex cleaving, vertex identification, and twisting. Therefore $H$ is in $\mathcal{G}$ so $G$ is a forbidden flat for $\mathcal{G}$. Conversely, let $G$ be a forbidden flat for $\mathcal{G}$ and let $F$ be its cycle matroid. Since $\mathcal{G}$ is matroid-closed, it is clear that $F$ is not in $\mathcal{M}$. Note that any proper flat of $F$ is the cycle matroid of a proper flat of $G$ and so in $\mathcal{M}$. Therefore $F$ is forbidden flat for $\mathcal{M}$. 
\end{proof}

We omit the proof of the following straightforward result.

\begin{lemma}
\label{hereditary_graph}
The edge-apex class, $\Ga$, of a flat-hereditary class $\mathcal{G}$ of graphs is flat-hereditary. Moreover, if $\mathcal{G}$ is matroid-closed, then so is $\Ga$.
\end{lemma}

The \textbf{rank} of a graph $H$ is the number of edges in a spanning forest of $H$. Note that the rank of $H$ equals the rank of the cycle matroid of $H$. 
Theorem \ref{bound_k} when translated for matroid-closed flat-hereditary graph classes says that if such a class of graphs has a finite number of forbidden flats, then so does its edge-apex class. The following is the precise statement.

\begin{theorem}
\label{bound_k_graph}
Let $\mathcal{G}$ be a flat-hereditary class of graphs that is closed under the operations of vertex cleaving, vertex identification, and twisting, and let $r$ denote the maximum rank of a forbidden flat for  $\mathcal{G}$, and $k$ denote the maximum number of edges in a forbidden flat for $\mathcal{G}$. If $G$ is a forbidden flat for $\Ga$, then the rank of $G$ is at most $\max \{2r, r + k(r-1)\}$. 
\end{theorem}

\begin{proof}
Consider the class $\mathcal{M}$ of matroids that are cycle matroids of the graphs in $\mathcal{G}$. By Lemma \ref{error_fixed}, the class $\mathcal{M}$ is hereditary and the graphic forbidden flats for $\mathcal{M}$ are the cycle matroids of the forbidden flats for $\mathcal{G}$. Since the rank of a forbidden flat for $\mathcal{G}$ equals the rank of its corresponding cycle matroid, a graphic forbidden flat for $\mathcal{M}$, it follows by Lemma \ref{error_fixed} that $r$ is the maximum rank of a graphic forbidden flat for $\mathcal{M}$ and $k$ is the maximum number of elements in a graphic forbidden flat for $\mathcal{M}$. Note that the class of cycle matroids of $\Ga$ is the graphic extension class $\Meg$ of $\mathcal{M}$ so by Corollary \ref{bound_k_corollary} and Lemmas \ref{error_fixed} and \ref{hereditary_graph}, the result follows.
\end{proof}

Examples of hereditary classes of graphs that are flat-hereditary and matroid-closed include the following.

\begin{proposition}
Let $\mathcal{G}$ be a hereditary class of graphs such that the set of the forbidden induced subgraphs for $\mathcal{G}$ is closed under the operation of twisting and each forbidden induced subgraph for $\mathcal{G}$ is $2$-connected. Then $\mathcal{G}$ is flat-hereditary and matroid-closed. 
\end{proposition}

\begin{proof}
By Lemmas \ref{prelim} and \ref{prelim_stronger}, it follows that $\mathcal{G}$ is closed under $0$-sum and $1$-sum. By Proposition \ref{prop_b}, the class $\mathcal{G}$ is flat-hereditary. Since $\mathcal{G}$ is hereditary and closed under $0$-sum and $1$-sum, it follows that $\mathcal{G}$ is closed under the operations of vertex cleaving and vertex identification. Suppose $G$ is a graph in $\mathcal{G}$ such that a graph $G'$ obtained from $G$ via twisting is not in $\mathcal{G}$. Since $G'$ is not in $\mathcal{G}$, the graph $G'$ contains a forbidden induced subgraph $H'$ for $\mathcal{G}$. Note that $G$ contains $H'$ as an induced subgraph or a subgraph $H$ as an induced subgraph where $H$ is a twisting of $H'$. Since both $H$ and $H'$ are forbidden induced subgraphs for $\mathcal{G}$, it follows that $G$ is not in $\mathcal{G}$, a contradiction.
\end{proof}

For a class $\mathcal{G}$ as in Theorem \ref{bound_k_graph}, since the rank of a forbidden flat $G$ for $\Ga$ is bounded, it follows that the number of vertices of $G$ is bounded. Therefore by Lemma \ref{flat_fis}, the number of vertices of a forbidden induced subgraph for $\Ga$ is bounded.
Using Theorem \ref{bound_k_graph},  we obtain the same bound as in Theorem \ref{unique_FIS} for hereditary classes of graphs that are matroid-closed and closed under $0$-sum. We require the following lemma.

\begin{lemma}
\label{point_wala}
Let $\mathcal{G}$ be a hereditary class of graphs closed under $0$-sum and let $G$ be a forbidden flat for the edge-apex class $\Ga$ of $\mathcal{G}$. If $G$ is not connected, then $G = G_1 \oplus G_2$ such that both $G_1$ and $G_2$ are isomorphic to a forbidden induced subgraph for $\mathcal{G}$.
\end{lemma}

\begin{proof}
Suppose that $G$ is disconnected, and let $G = G_1 \oplus \ldots \oplus G_k$. Suppose there is a component, say $G_1$, that does not contain a forbidden induced subgraph for $\mathcal{G}$. Then $G_1$ is in $\mathcal{G}$. Consider $G' = G_2 \oplus \ldots \oplus G_k$. Observe that $G'$ is in $\Ga$ but not in $\mathcal{G}$. It follows that $G'$ has an edge $e$ such that $G'-e$ is in $\mathcal{G}$. Since $G-e = G_1 \oplus G'-e$, and both $G_1$ and $G'-e$ are in $\mathcal{G}$, it follows that $G-e$ is in $\mathcal{G}$ so $G$ is in $\Ga$, a contradiction.  
If $G$ has more than two connected components, then $G' = G_2 \oplus \ldots \oplus G_k$ is a proper induced subgraph of $G$ that has at least two connected components. Since each connected component of $G'$ contains a forbidden induced subgraph for $\mathcal{G}$, it follows that $G$ is not in $\Ga$, a contradiction. Therefore $G = G_1 \oplus G_2$. Suppose that the forbidden induced subgraph $F$ contained in, say, $G_1$ is a proper induced subgraph of $G_1$. Then there is a vertex $v$ of $G_1$ such that $F$ is an induced subgraph of $G_1 - v$. Note that the graph $G-v = (G_1 - v) \oplus G_2$ contains two edge disjoint forbidden induced subgraphs for $\mathcal{G}$ so is not in $\Ga$. This is a contradiction to the minimality of $G$. Therefore both $G_1$ and $G_2$ are isomorphic to forbidden induced subgraphs for $\mathcal{G}$.
\end{proof}

\begin{corollary}
\label{final_conclusion}
Let $\mathcal{G}$ be a hereditary class of graphs that is closed under the operations of $0$-sum, vertex identification, and twisting, and let $c$ and $k$ denote the maximum number of vertices and edges, respectively, among all forbidden induced subgraphs for $\mathcal{G}$. If $G$ is a forbidden induced subgraph for $\Ga$, then $|V(G)| \leq $  $\max \{2c, c + k(c-2)\}$.
\end{corollary}

\begin{proof}
By Proposition \ref{prop_b}, it follows that $\mathcal{G}$ is flat-hereditary and every forbidden induced subgraph for $\mathcal{G}$ is a forbidden flat for $\mathcal{G}$. To bound the number of vertices of $G$, we will first bound the vertices of an arbitrary forbidden flat $F$ for $\Ga$. Since $\mathcal{G}$ is closed under $0$-sum, by Lemma \ref{prelim}, all forbidden induced subgraphs for $\mathcal{G}$ are connected. It follows that the maximum rank of a forbidden flat for $\mathcal{G}$ is $c-1$, and $k$ is the maximum number of edges in a forbidden flat for $\mathcal{G}$. Note that $\mathcal{G}$ is closed under vertex cleaving so by Theorem \ref{bound_k_graph}, every forbidden flat $F$ for $\Ga$ has rank at most $\max \{2(c-1), (c-1) + k(c-2)\}$.

Note that if $F$ is connected, then $|V(F)| \leq 1+ \max \{2(c-1), (c-1) + k(c-2)\} = \max \{2c-1, c + k(c-2)\}$. Alternatively, if $F$ is disconnected, then by Lemma \ref{point_wala}, $|V(F)| \leq 2c$. Therefore $F$ has at most $\max \{2c, c + k(c-2)\}$ vertices. Finally, by Lemma \ref{flat_fis}, every forbidden induced subgraph $G$ for $\Ga$ has the same number of vertices as some forbidden flat for $\Ga$, so the result follows.
\end{proof}

\section*{Acknowledgement}

The authors thank James Oxley for helpful suggestions.

\end{document}